\documentclass[10pt,a4paper]{article}
\usepackage{amssymb,amsmath,amsthm}

\theoremstyle{plain}
\newtheorem{theorem}{Theorem}
\newtheorem{lemma}{Lemma}

\newcommand\base[1]{\widehat{#1}}
\newcommand\m{{\bar m}}
\newcommand\tdue{\mathfrak{t}_2}
\newcommand\tuno{\mathfrak{t}_1}
\newcommand\tres{\mathfrak{t}_\mathrm{res}}
\newcommand\tzero{\mathfrak{t}_0}
\newcommand\Tzero{T_0}
\newcommand\Tuno{T_1}
\newcommand\Tdue{T_2}

\makeatletter
\renewcommand{\@oddhead}{\it
  Fass\`o, Garc\'{\i}a-Naranjo and Giacobbe: 
  Relative quasi-periodic tori 
\hfill \thepage}
\makeatother

\newcommand{\g}{\mathfrak{g}}

\newcommand{\for}[1]{(\ref{#1})}

\newcommand{\p}{\varphi}
\renewcommand{\P}{\Phi}

\renewcommand{\a}{\alpha}
\renewcommand{\d}{\delta}
\renewcommand{\o}{\omega}

\renewcommand{\L}{\Lambda}

\newcommand{\cF}{\mathcal{R}}                 

\newcommand{\cP}{\mathcal{T}}                 

\newcommand{\cT}{\mathcal{T}}                 

\newcommand{\cX}{\mathcal{X}}

\newcommand{\acapo}{\vspace{2ex}}
\newcommand{\acapon}{\vspace{2ex}\noindent}

\newcommand{\toro}{\mathbb{T}}

\newcommand{\reali}{\mathbb{R}}

\newcommand{\interi}{\mathbb{Z}}

\newcommand{\ug}{\;=\;}

\newcommand{\per}{\times}		
\renewcommand{\bar}{\overline}		

\newcommand{\const}{\mathrm{const}} 

\newcommand\rank{\mathrm{rank\,}}

\renewcommand{\mod}{\mathrm{mod}} 

\newcommand\bList{
\begin{list}{}{\leftmargin2em\labelwidth1.5em\labelsep.5em\itemindent0em
\topsep.3ex\itemsep-.4ex} }
\newcommand\eList{\end{list}}

\title{\bf Quasi-periodicity \\
in relative quasi-periodic tori}

\begin{document}
\author{
Francesco Fass\`o\footnote{\scriptsize Universit\`a di Padova,
Dipartimento di Matematica, Via Trieste 63, 35131 Padova, Italy.
Email: \tt{fasso@math.unipd.it}, \tt{giacobbe@math.unipd.it}
}  , \ 
Luis C. Garc\'{\i}a-Naranjo\footnote{\scriptsize 
Departamento de Matem\'aticas y Mec\'anica, 
IIMAS-UNAM, 
Apdo Postal 20-726,  Mexico City,  01000, Mexico. Email:
\tt{luis@mym.iimas.unam.mx}
}\ \ 
and Andrea Giacobbe$^*$
}
\maketitle

\begin{abstract}
\noindent
At variance from the cases of relative equilibria and relative
periodic orbits of  dynamical systems with symmetry, the dynamics in
relative quasi-periodic tori (namely, subsets of the phase space
that project to an invariant torus of the reduced system on which the
flow is quasi-periodic) is not yet completely understood. Even in the
simplest situation of a free action of a compact and abelian
connected group, the dynamics in a relative quasi-periodic torus is
not necessarily quasi-periodic. It is known that quasi-periodicity of
the unreduced dynamics is related to the reducibility of the
reconstruction equation, and sufficient conditions for it are
virtually known only in a perturbation context. 
We provide a
different, though equivalent, approach to this subject,
based on the hypothesis of the existence of commuting,
group-invariant lifts of a set of generators of the reduced torus.
Under this hypothesis, which is shown to be equivalent to the
reducibility of the reconstruction equation,
we give a complete description of
the structure of the relative quasi-periodic torus, which is a
principal torus bundle whose fibers are tori of a dimension which
exceeds that of the reduced torus by at most the rank of the group.
The construction can always be done in such a way that these tori
have minimal dimension and carry ergodic flow. 
\end{abstract}


\section{Introduction}
\label{s:1}

\subsection{Relative quasi-periodic tori } 
\label{s:1.1}

In a dynamical system with symmetry, a group-invariant subset of the
phase space is called a {\it relative equilibrium} if its quotient
with respect to the symmetry group $G$ is an equilibrium of the
reduced flow, and a {\it relative periodic orbit} if its quotient is
a periodic orbit of the reduced flow. More generally, a {relative
set} is a group-invariant subset whose quotient is invariant under
the reduced dynamics.\footnote{Sometimes a different convention is
used, and the terms relative equilibria and relative periodic orbits
denote the individual orbits in such relative sets, rather than the
relative sets themselves.} 

The dynamics in relative equilibria and periodic orbits is well
understood. It has been proved by Field \cite{field1,field2} and
Krupa \cite{krupa} that, under hypotheses which include the case of a
free action of a compact and connected group, 
the flow in these two types of relative sets is conjugate to
a linear flow on a torus (or, as we shall say, is `quasi-periodic') of
dimension  not greater than the rank of the group (for relative
equilibria) or of the rank of the group plus one (for relative
periodic orbits). 
The dynamics in relative equilibria and periodic orbits in the case
of a non-compact group, that need not be quasi-periodic, is well
understood as well~\cite{ashwin-melbourne}. Details on these results,
and their modifications under more general hypotheses, can be found
in the quoted references and in the recent books \cite{field3,CDS}.
The related problem of reconstructing families of relative periodic
orbits (e.g., when the entire reduced dynamics is periodic) has been
studied in \cite{hermans} in the case of a free action of a compact
group (see also \cite{fasso-giacobbe, CDS}).

The situation is different if the reduced dynamics is
quasi-periodic. 

By a {\it relative quasi-periodic torus} we mean a
relative set whose quotient is an invariant torus of the reduced
system, that has dimension $\ge2$ and on which the reduced flow is conjugate
to a linear flow. Clearly, it may be assumed that the reduced flow
has nonresonant frequencies, so that the reduced invariant torus is
the closure of any reduced trajectory on it (if not, restrict to a
subtorus). It is well
known that, even in the simplest case of a compact and connected group, the
dynamics in a relative quasi-periodic torus need not be
quasi-periodic. For a simple example see \cite{Zung}, where this
failure is interpreted in terms of the 
appearance of `small divisors' in the reconstruction process.

{\it{Per se}}, the specific topic of quasi-periodicity in relative
quasi-periodic tori has so far received little attention. The only
papers we know are \cite{zenkov,Zung}. Reference \cite{Zung} proves
that, under natural hypotheses, relative quasi-periodic tori of
{Hamiltonian} systems have quasi-periodic dynamics. Therefore, this
problem is of interest only in the non-Hamiltonian case.
Reference \cite{zenkov}, even though considers only a special
non-Hamiltonian case, lays
down ideas that apply in general and links the problem 
to the theory of reducibility of quasi-periodic linear equations, 
a subject on which there is a vast literature (see e.g. \cite{puig}
and references therein). Specifically, reference \cite{zenkov} proves that,
under natural hypotheses that we will review in Section
\ref{s:Comparison}, the reducibility of the
reconstruction equation implies the quasi-periodicity of the dynamics
in the relative quasi-periodic torus. 

However, reducibility of a quasi-periodic linear system is a difficult
matter. There exist sufficient conditions of perturbative KAM
type, but they require
hypotheses (strong nonresonance and the presence of parameters that
may be varied, see e.g. \cite{broer1,broer2}) that are hardly
satisfied in the case of a single relative torus. More easily
applicable to this case are results of approximate reducibility,
called effective or almost reducibility (see e.g.
\cite{jorba,chavaudret}), that require weaker conditions
but give only closeness of the  dynamics to quasi-periodicity for
finite though long time scales \cite{zenkov}.

At a deeper level, reference \cite{zenkov} proves that the
reconstruction equation is reducible if and only if the unreduced
vector field has an additional torus symmetry, a fact that will
play a role in our approach as well.

\subsection{Our results} 
\label{s:1.2}
The aim of this article is to introduce a new, though equivalent to
reducibility,
point of view on the subject of reconstruction of
quasi-periodic dynamics. We focus on the simplest case of a free
action of a compact and connected group and develop a reconstruction
technique from reduced quasi-periodic dynamics that generalizes the
known procedure used to reconstruct reduced periodic orbits, that
is described in \cite{field2, krupa,
hermans, field3,CDS}. 

If the reduced flow is
quasi-periodic with nonresonant frequencies, then no reduced orbit
ever returns to the same point and the unreduced (or, as we shall
say, the `reconstructed') orbits never return to the same group
orbit. This makes it impossible to define the basic object used in
the reconstruction theory for periodic orbits---the group element
associated to the first return map and variously called either `phase' or
`monodromy' or `shift'. 

We overcome this difficulty by applying the reconstruction techniques
for periodic orbits not to the flow, but to a homology basis of the
reduced torus. In order to do this, we need the assumption that, if
the reduced torus is $k$-dimensional, a set
of $k-1$ of its generators lifts to a set of group-invariant
vector fields that commute among each other and with the dynamics.

Under this hypothesis, we prove that the {relative quasi-periodic
torus} has a structure analogous to that of relative equilibria and
periodic orbits: it is fibered by invariant tori of a certain
dimension $k+d$, with $d$ between zero and the rank of the group, on
which the flow is linear (Theorem 1). The frequencies of the
reconstructed flow are the $k$ `internal' frequencies of the reduced
flow plus $d$ other `external' frequencies that are identified in the
reconstruction process in terms of the `phases' of the lifts of the
generators and of the Lie algebra of a certain torus in the group,
whose construction depends on both geometric and dynamic properties
of the system.\footnote{The terms `internal' and `external'
frequencies are used in reducibility theory, see e.g. \cite{puig}.}

Furthermore, as in the case of relative periodic orbits
\cite{field1,CDS}, this construction can always be done in a way
which is natural from the dynamical point of view as well, that is,
the dimension $k+d$ of the reconstructed tori is `minimal' and the
flow is dense on them; this may require the consideration
of a covering of the relative torus, and subharmonics of the reduced
frequencies may appear (Theorem 2).

We also prove that our hypothesis is equivalent to the
reducibility of the reconstruction equation (Theorem 3). Thus, our
results on quasi-periodicity can be regarded as a different
characterization of reducibility, with the addition, however, of a 
detailed description of the geometry of the 
relative quasi-periodic torus, which was missing so far. 

\subsection{Comments} 
\label{s:1.3}
Sample systems to which our technique may be applied can easily 
be constructed (see e.g. Section \ref{s:3.2}). It would
be very interesting if our technique turned out to be applicable to
some `natural' problem, and in particular to nonholonomic systems, to which
the well understood machinery of the Hamiltonian mechanics does not
apply. Certainly, proving the existence of lifts with the desired
properties may be difficult in practice but, as we will illustrate on
some examples, it is not more difficult than proving reducibility. 

Both, our hypotheses and the reducibility of the reconstruction
equation, are sufficient conditions for quasi-periodicity of the
dynamics. The reason why they are not also necessary is related to
global aspects. The reconstruction equation, in its usual formulation,
assumes that the relative set has the product structure $G\times
\toro^k$, but it is certainly possible that the dynamics be
quasi-periodic in a relative torus which is a nontrivial bundle with
fiber $G$ and base $\toro^k$. In our characterization, this
triviality assumption is hidden in the fact that we assume the
existence of global lifts of the generators, an assumption that might
be weakened. We defer these generalizations to a future work. 

Generalizations to non-connected groups and non-free actions can be
treated easily, much in the same way as in the case of relative
periodic orbits, and do not introduce important differences, see
Section \ref{s:comments1}. The case of non-compact groups is
significantly different, and would require a distinct analysis.

Finally, we note that a natural extension of the present study concerns
the case in which the reduced dynamics has not just one relative
quasi-periodic torus, but is integrable, in the sense that it has
quasi-periodic dynamics. This is indeed the case considered in
references \cite{zenkov,Zung} when the reduced dynamics is
quasi-periodic and in references \cite{hermans,fasso-giacobbe,CDS}
when the dynamics is periodic. Our technique is applicable to
study such a case but, because of certain specificities of this
problem---which are related to the global structure of the fibration
by invariant tori---we prefer deferring this study to a separate work.

\section{Reconstruction theorems }
\label{section2}

\subsection{Statements}
\label{s:2.1}
Consider a free action $\Psi$ of a compact and connected Lie group $G$ on a
manifold $M$. Denote by $\base{M}$ the quotient manifold $M/G$ and by
$\pi:M\to\base{M}$ the canonical projection. Let $X$ be a
$G$-invariant vector field on $M$.
The reduced vector field is the vector field
$\base{X} : =\pi_*X$ on $\base{M}$.\footnote{The push-forward $\pi_*X$
of $X$ is the vector field on $\base M$ whose value in $\base m \in
\base M$ is 
$\base{X}(\base{m}) = T_m \pi \cdot X(m)$ for any $m\in\pi^{-1}(\base m)$.} 

Assume also that the reduced vector field $\base{X}$ has an invariant
torus on which its flow is quasi-periodic. Specifically, there is an
embedding
$$
   i : \toro^k = (\reali/\interi)^k\to \base{M}
$$
such that $\base{X}$ is tangent to $\base{\cF}:=i(\toro^k)$ 
and, denoting by $\p=(\p_1,\ldots,\p_k)$ the angles in
$\toro^k$,\footnote{The pull-back $i^*(\base X|_{\base{\cF}})$ is the 
unique vector field on $\toro^k$ such that $Ti\cdot i^*(\base
X|_{\base{\cF}})= \base X\circ i$.}
$$
 i^*(\base X|_{\base{\cF}}) = \sum_{j=1}^k\o_j \partial_{\p_j}
$$
with $\o_1,\ldots,\o_k\in\reali$.
We assume that the frequencies
$\o_1,\ldots,\o_k$ are nonresonant, that is, independent over
$\interi$ (if they are not, then the reconstruction applies to each
nonresonant subtorus of~$\base\cF$). Hence, the closure of each
reduced motion in $\base{\cF}$ is the entire $\base{\cF}$. 

From now on we restrict our analysis to the {\it relative
quasi-periodic torus}
$$
  \cF:=\pi^{-1}(\base{\cF}) \,,
$$
which is a compact submanifold of $M$ of dimension $k+\dim G$ and is
$G$-invariant and $X$-invariant. The restriction $\pi|_\cF$
of $\pi$ to $\cF$ makes such a submanifold a $G$-principal bundle over
$\base{\cF}$. From now on we will
write $X$ for $X|_\cF$, $\pi$ for $\pi|_\cF$ etc. (The 
manifold $M$ will not play any role in the sequel; in fact, we might even
have taken $M=\cF$).

Recall that a dynamical symmetry of a vector field $Y$ is any vector
field that commutes with $Y$. Moreover, we say that a vector field $S$ on
$\cF$ is a {\it lift} of a vector field $\base{S}$ on
$\pi(\cF)=\base\cF$ if
$\pi_*S=\base{S}$. As usual, we denote  by $\cX(\cF)$ the set of all
vector fields on $\cF$ and by $\P^Y_t$ the time-$t$ map of
the flow of a vector field $Y$. 

With a 
little abuse we denote by
$\partial_{\p_1},\ldots,\partial_{\p_k}$
the vector fields $i_*(\partial_{\p_1}),\ldots, i_*(\partial_{\p_k})$ in
$\base\cF$, and we call them {\it generators} of $\base\cF$. 
Of course, the generators are not unique: any choice of angles on
$\widehat\cF$, namely of the embedding $i$, gives a set of them.

\begin{theorem}\label{th1}
Under the stated hypotheses, assume that there exist lifts of $k-1$
among a set of generators of $\base\cF$ which
are $G$-invariant, are dynamical symmetries of $X$, and pairwise commute.
Then:
\bList
\item[i.] There exists a fibration of $\cF$ whose fibers are $X$-invariant
manifolds diffeomorphic to $\toro^{k+d}$, for a certain $0\le
d\le\rank G$, on which the flow of $X$ is conjugate to the
linear flow on $\toro^{k+d}$ with frequencies
$(\o_1,\ldots,\o_k,\nu_1,\ldots,\nu_d)$, with
$(\nu_1,\ldots,\nu_d)\in\reali^d$ constant on $\cF$.
\item[ii.] 
The base of this fibration of $\cF$ is diffeomorphic to $G/T$, where
$T\subset G$ is a $d$-dimensional torus. 
\eList
\end{theorem}

The above fibration, including the dimension $d$ of the tori, is in
general not unique. It will appear from the proof that, if $d<\rank
G$, then the same result is valid with tori of any dimension
$d+1,\ldots,\rank G$. In this statement it is not granted that the
frequencies $(\o,\nu)$ are nonresonant, in which case there might
exist fibrations of $\cF$ by tori of dimension smaller than such a $d$.
It is however always possible to obtain a similar
result with nonresonant frequencies, and hence tori of the minimal
dimension and as such uniquely defined. In doing so, however, one
realizes that the internal frequencies of the reconstructed motions might be
subharmonics of those of the reduced motions:

\begin{theorem}\label{th2}
Under the hypotheses of Theorem \ref{th1}:
\bList
\item[i.] There exist integers $0\le d_0\le\rank G$ and $r>0$, a
vector $\nu\in\reali^{d_0}$, a fibration of $\cF$ whose fibers
$\cT_m$ are $X$-invariant manifolds diffeomorphic to
$\toro^{k+{d_0}}$, and a covering map
$c_m: \toro^{k+{d_0}}\to\cT_m$ that relates the linear flow on
$\toro^{k+{d_0}}$ with frequencies $(\frac\o r,\nu)$, which are
nonresonant, to the flow of $X$ in $\cT_m$. 
\item[ii.] The base of this fibration of $\cF$ is diffeomorphic to $G/T$,
where $T$ is an abelian subgroup of $G$ isomorphic to $T_0\times F_0$, with
$T_0$ a ${d_0}$-dimensional torus of $G$ and $F_0$ a finite abelian
subgroup of $G$.
\eList
\end{theorem}

The proofs of the two theorems are given in Section \ref{s:proof}.

\subsection{Comments on Theorem 1}
\label{s:comments1}
What Theorem \ref{th1} actually deals with is not the preimage under
the canonical projection of a single quasi-periodic orbit, but of the
entire invariant torus. This is why we prefer the expression `relative
quasi-periodic torus' to `relative quasi-periodic orbit'.

Theorem \ref{th1} may be regarded as an extension to the
quasi-periodic case of the results for relative periodic orbits
\cite{field2,krupa,hermans,field3,CDS}. In the case of relative
periodic orbits $k=1$, and the hypotheses of the theorem yield no
conditions and are hence always satisfied.

We remark that under the hypotheses of Theorem \ref{th1}, there are
not only $k-1$, but $k$ pairwise commuting, $G$-invariant
dynamical symmetries of $X$ that are lifts of the generators
$\partial_{\p_1},\ldots,\partial_{\p_k}$. In fact, let
$S_1,\ldots,S_{k-1}$ be lifts of, say,
$\partial_{\p_1},\ldots,\partial_{\p_{k-1}}$, that have these
properties and define
\begin{equation}\label{ultimafase}
  S_k=\frac1{\o_k}\Big(X-{\sum_{j=1}^{k-1}} \o_j S_j\Big) \,;
\end{equation}
then $\pi_*S_k=\partial_{\p_k}$ and  $S_1,\ldots,S_k$ are lifts of
$\partial_{\p_1},\ldots,\partial_{\p_k}$ with these same properties.
In the proof, we will use such a set of $k$ lifts. 

Clearly, the lifts $S_1,\ldots,S_k$ define an action of $\reali^k$
that commutes with the action of $G$ and leaves $X$ invariant. To some
extent, the proof of Theorem \ref{th1} consists in showing
that this action is in fact an action of $\toro^k$. Therefore, under
the hypotheses of Theorem \ref{th1}, the relative quasi-periodic
torus $\cF$ is invariant under an action of the bigger group
$\toro^k\times G$. We will come back on this in Section
\ref{s:3.1} (Theorem 3) and in the Remark at the end of Section
\ref{s:proof3}.

There are two rather natural generalizations of Theorem 1 to
non-connected Lie groups and to non-free actions. 
\bList
\item[1.] Non-connected Lie
groups are of the form $G = H\rtimes K$, with $H$ the connected
component of the identity and $K$ a finite group. If the hypotheses of
Theorem 1 are satisfied, then they are also satisfied with the
$G$-action replaced with that of its normal subgroup $H$. The torus
$\widehat{\cF}$ embedded in $M/G$ is covered, with fiber
possibly a subgroup of $K$, by a torus embedded in $M/H$. The space
$\cF$ is the preimage of this last torus, and the fibration of
$\cF$ by tori can be obtained precisely with the same
arguments.

\item[2.] If the $G$ action is non-free, choose a point $m\in \cF$
and let $H$ be its stabilizer. Consider $\cF_H$, the
submanifold of all points of $\cF$ with orbit type $H$
(see \cite[pag. 16]{audin}). The vector field $X$ and the lifts
$S_1,\ldots,S_k$
described in Section \ref{s:2.1} are tangent to $\cF_H$, and the
phases belong to $N(H)/H$ (where $N(H)$ is the normalizer of $H$).
The hypotheses of Theorem 1 apply to the manifold $\cF_H$ under
the free action of $N(H)/H$, and the theorem provides a foliation of
$\cF_H$ in tori of dimension $k+d$ with $d$ at most the rank of
$N(H)/H$ (which is less that the rank of $G$). The foliation in tori
can then be induced in the whole manifold $\cF$ using the $G$
action. 
\eList

\vskip2mm\noindent
{\it Remark: } The proof of 
Theorem \ref{th1} is trivial when $G$ is abelian. A compact abelian group is
a torus, hence $G \cong \toro^d$, and one can use independent
infinitesimal generators of the torus action as additional dynamical
symmetries that commute with $X$ and with $S_1, \ldots, S_k$; so, $X$
is a $\toro^{k+d}$-invariant vector field in $\toro^{k+d}$.

\subsection{Comments on Theorem 2}
\label{s:comments2}

The reason why, in Theorem \ref{th2}, the reconstructed motions may
have the first $k$ frequencies that are submultiples of those
$\o_1,\ldots,\o_k$ of the reduced motion can be understood on a
simple example. The vector field
$$
  X = \partial_{\a_1}+\sqrt2\,\partial_{\a_2}+\frac12\,\partial_{\a_3}
$$
on $\cF=\toro^3\ni(\a_1,\a_2,\a_3)$ is equivariant under the action
of $G=S^1$ by translations of the third angle. The reduced system is
the vector field $\widehat X = \partial_{\a_1}+\sqrt2\partial_{\a_2}$
on $\base\cF=\toro^2\ni(\a_1,\a_2)$, so $k=2$ and
$(\o_1,\o_2)=(1,\sqrt 2)$. A possible reconstruction would obviously
add back the third angle, leading to (a single) reconstructed torus
of dimension $3$, with internal frequencies $\o=(1,\sqrt2)$, $d=1$
and external frequency $\nu=\frac 12$. But the dynamics of the
unreduced system is dense in the 2-dimensional subtori of
$\cF=\toro^3$ given by $\a_1-2\a_3=\const$. This suggests that it
should also be possible to perform the reconstruction process so as to
have ${d_0}=0$ and a fibration of $\cF=\toro^3$ by two-dimensional
reconstructed tori, with no external frequencies. In fact, in the
coordinates $\p_1=\a_3$, $\p_2=\a_2$, $\p_3=\a_1-2\a_3$ on  $\cF$,
the unreduced vector field is
$$
  X = \frac12\,\partial_{\p_1} + \sqrt2\,\partial_{\p_2}
$$
and its flow is conjugate to the linear flow on the  two-dimensional
tori $\p_3=\const$ with frequencies $(\frac12,\sqrt 2)=
(\frac{\o_1}2,\o_2)$. 

In Theorem \ref{th2}, the dimension $k+{d_0}$ of the reconstructed
torus may be anything between $k$ and $k+\rank \, G$. Mathematical
examples of all these situations may of course be easily constructed.
A very interesting example  coming from nonholonomic mechanics is an 
$n$-dimensional generalization of the classical Veselova system
\cite{veselova1,veselova2} considered in
\cite{fedorov-jovanovic}. Such problem describes the motion of an
$n$-dimensional rigid body with fixed point subject to a nonholonomic
constraint. The phase space of the system is a vector subbundle
$\mathcal D$ of $TSO(n)$ of rank $n-1$ (and hence a manifold of
dimension $n(n-1)/2+(n-1)$), and the systems has an ${\rm SO}(n-1)$
symmetry-group. The reduced dynamics takes place in $\mathcal
D/SO(n-1)$ (diffeomorphic to $TS^{n-1}$) and, up to a time
parametrization, is quasi-periodic in $(n-1)$-dimensional tori. Using
techniques different from those used here, Fedorov and Jovanovic
\cite{fedorov-jovanovic} prove that also the unreduced dynamics is,
in the new time variable, quasi-periodic in tori of dimension $n-1$
which foliate the phase space $\mathcal D$. 
If this case could be analyzed in the realm of our
approach, complete resonance between the `internal' and `external'
frequencies $\omega$ and $\nu$ would be found, making
${d_0}=0$.\footnote{A personal communication with Y.N. Fedorov
suggests that (part of) the frequencies of the unreduced motions are
subharmonics of those of the reduced motions.}

\section{Comparison with reducibility }
\label{s:Comparison}

\subsection{Reducibility and its equivalence to our hypotheses}
\label{s:3.1}
We prove here the equivalence of our approach and of that based on
the reducibility of the reconstruction equation, considered by Zenkov
and Bloch \cite{zenkov}. We adapt
the treatment in \cite{zenkov}, that considers a specific problem\footnote{A
nonholonomic mechanical system with an $\mathrm{SO}(n)$ symmetry,
which is an $n$-dimensional generalization of the classical Suslov
problem \cite{suslov,fedorov-kozlov} whose reduced dynamics is
quasi-periodic}, to our general setting. 

We consider a situation similar to that of Section \ref{s:2.1},
with a free action
$\Psi$ of a compact and connected Lie group $G$ on a manifold $\cF$
and $\base \cF = \cF/G$ diffeomorphic to $\toro^k$. 
As is typical in reducibility theory, we assume that $\cF\to\cF/G$
is a trivial $G$-bundle, so that there is an equivariant
diffeomorphism from $\cF$ onto $\toro^k\times G$. (In the sequel, 
`equivariant' means always equivariant with respect to the
$G$-action $\Psi$ on $\cF$ and the $G$-action on $\toro^k\times
G$ by left translations on the factor $G$). Furthermore,
we consider a $G$-invariant vector field on $\cF$, whose reduced flow
on $\base\cF$ is conjugate to a linear flow on $\toro^k$ with
frequencies $\o=(\o_1,\ldots,\o_k)$.   
Written in
$\toro^k\times G \ni (\varphi,g)$, the 
equations of motion of $X$ on $\cF$
take the form of the linear skew-product system
\begin{equation}\label{triv}
  \dot \varphi = \omega \,,\qquad 
  \dot g = g \xi(\varphi) 
\end{equation}
for some $\xi:\toro^k\to \mathfrak g$, the Lie algebra of $G$.
(For simplicity, we assume here
that the group is a matrix group so that $L_g\xi=g\xi$).
 The partially integrated equation
\begin{equation}\label{rec-eq}
  \dot g = g \xi(t\omega) 
\end{equation}
is then called the reconstruction equation for $X$. 

The reconstruction equation is said to be reducible if there exists a
map $a:\toro^k\to G$ such that the diffeomorphism
\begin{equation}\label{diffeo}
   (\varphi,g)\mapsto 
   \big(\varphi,g a(\varphi)^{-1} \big) =:(\varphi,h) 
\end{equation}
of $\toro^k\times G$ onto itself conjugates system \for{triv} to the system
\begin{equation}\label{triv2}
  \dot \varphi = \omega \,,\qquad 
  \dot h = h \rho
\end{equation}
with a constant $\rho\in \mathfrak g$.
This happens if and only if, writing for shortness
$\frac d{dt} a(\varphi)$ for
$\frac{\partial a}{\partial \varphi}(\varphi)\omega$, the map
$a$ is such that 
\begin{equation}\label{eq:reducibility}
   a(\varphi) \xi(\varphi) a(\varphi)^{-1}  
   - 
   \big[ \textstyle{\frac d{dt}} a(\varphi) \big] a(\varphi)^{-1}  
   \quad\mathrm{is \ constant} \,.
\end{equation}
Then $\rho$ equals such a constant and
the flow of equations \for{triv2} is quasi-periodic on tori of
dimension at most $k+\rank G$. 

We remark that while the reconstruction equation is always
reducible if the time dependence is periodic (Floquet theory), its
reducibility is generally unknown in the quasi-periodic case
(see e.g. \cite{puig} for a review of reducibility theory with
emphasis on the quasi-periodic case).

Reference \cite{zenkov} proves that the reducibility of the reconstruction
equation is equivalent to the invariance of the system under an extra
$\toro^k$-action.  We use this fact to prove the equivalence between
our hypotheses and reducibility.

\begin{theorem} Assume that a vector field $X$ on a compact manifold
$\cF$ is invariant under a free action $\Psi$ of a compact and connected
Lie group $G$. Assume that $\widehat\cF=\cF/G$ is diffeomorphic to a
$k$-dimensional torus and that the flow of the reduced vector field
$\widehat X$ on $\widehat\cF$ is quasi-periodic with nonresonant
frequencies. Then, the following three conditions are equivalent:

\bList
\item[i.] $\cF\to\cF/G$ is a trivial $G$-bundle 
and the reconstruction equation is reducible.
\item[ii.] 
There exists a free action $\chi$ of $\toro^k$ on $\cF$ that leaves
$X$ invariant, commutes with $\Psi$ and is such that the action
$(\chi,\Psi)$ of $\toro^k\times G$ on $\cF$ is free.
\item[iii.] There exist $G$-invariant lifts of $k-1$ among a set of
generators of
$\widehat{\cF}$ that commute among themselves and with $X$.
\eList
\end{theorem}

The proof of this theorem uses some constructions from the proof of
Theorem 1 and is postponed to Section \ref{s:proof3}.

\acapo
{\it Remarks:} 1. Reference \cite{zenkov} proves, in a particular case, 
the equivalence of the two conditions i. and
ii. of Theorem 3. However, some details are missing in that reference.
In particular, the hypothesis of the triviality of the $G$-bundle $\cF$
is only implicit (without it, it is not possible to write the
equations of motion in the form \for{triv} and therefore the
reconstruction equation \for{rec-eq}). Also, the freeness of the joint
$(\chi,\Psi)$ action in condition ii. is not noticed. 

2. In condition ii. of Theorem 3, the existence of an action of
$\mathbb T^k$ could be equivalently
replaced by the existence of an action of $\mathbb T^{k-1}$.

\subsection{Examples}
\label{s:3.2}

We compare now our approach with reducibiliy in two examples.

\vskip2mm\noindent
{\it Example 1. } Consider the vector field
\begin{equation*}
  X =\partial_{\varphi_1} +\sqrt{2} \partial_{\varphi_2} +
  \big(f_1(\varphi_1)+f_2(\varphi_2)\big) \xi_\mathrm{SO(3)} (g) 
\end{equation*}
on $\cF=\toro^2\times \mathrm{SO(3)} \ni
(\varphi_1,\varphi_2,g)$,
where $f_1, f_2 \in C^\infty(S^1)$ and $\xi\in \mathfrak{so}(3)$
($\xi_\mathrm{SO(3)}(g)\in T_g\mathrm{SO(3)}$ is the value at $g$
of the infinitesimal generator of $\xi$
associated to left multiplication). The group
$G=\mathrm{SO(3)}$ acts by left multiplication on the $\mathrm{SO(3)}$ factor
of $\cF$ and leaves $X$ invariant. The reduced space is
$\widehat {\cF}=\toro^2$ and the reduced vector field
\begin{equation*}
\widehat X =\partial_{\varphi_1} +\sqrt{2} \partial_{\varphi_2}
\end{equation*}
has quasi-periodic flow with frequencies
$(\omega_1,\omega_2)=(1,\sqrt{2})$. The vector field
\begin{equation*}
 S_1=\partial_{\varphi_1} + f_1 \xi_\mathrm{SO(3)}
\end{equation*}
is a lift of $\partial_{\varphi_1}$, is $\mathrm{SO(3)}$-invariant,
and commutes with $X$. 
Therefore, given that $\mathrm{SO(3)}$ has rank one, Theorem 1
(Theorem 2) 
implies that the flow of $X$ on $\cF$ is conjugate to a (nonresonant)
linear flow on tori
of dimension either two or three (which possibility is realized
depends on $f_1,f_2,\xi$). 

To investigate the reducibility of the system write
$f_1(\varphi_1)+f_2(\varphi_2)= \tilde f_1(\varphi_1)+ \tilde
f_2(\varphi_2)+\nu$ where $\tilde f_1$ and $\tilde f_2$ have zero
averages and $\nu\in \mathbb R$ is the sum of the averages of $f_1$ and $f_2$. 
The reconstruction equation is $\dot g = g \Omega(t,\sqrt2\,t)$
where
$$
  \Omega(\varphi_1,\varphi_2) =
  ( \tilde f_1 (\varphi_1) + \tilde f_2 (\varphi_2) +\nu)\xi  \,.
$$
A fundamental matrix of the reconstruction equation is 
$$
  A(t) 
  =
  \exp\left (\xi \int_0^t \Omega(s) \, ds \right ) 
  = 
  R\big( F_1(t) + 2^{-1/2}F_2(\sqrt{2}t)+ \nu t\big) 
$$
where $F_1, F_2\in C^\infty(S^1)$ are the primitives of $\tilde f_1,
\tilde f_2$ that vanish at $0$ and $R(x)$ is the rotation matrix
$\exp(x\xi)$.
Notice that $A(t)$ factors as $A(t)=\exp(t\nu \xi)a(t,\sqrt2\,t)$ where
$$
 a(\varphi_1,\varphi_2)
 =
 R\big( F_1(\varphi_1) + 2^{-1/2} F_2(\varphi_2)\big) 
$$
satisfies the reducibility equation \for{eq:reducibility}.

\vskip2mm\noindent
{\it Example 2. } Consider the vector field $X$ on $M=\toro^k\times
\toro^d\ni (\varphi,\vartheta)$
\begin{equation*}
\label{E:X}
  X=\sum_{i=1}^k \omega_i \partial _{\varphi_i} + \sum_{I=1}^d \xi_I (\varphi) 
   \, \partial _{\vartheta_I}
\end{equation*}
with nonresonant $\o_1,\ldots,\o_k$.
$X$ is invariant under the action of $G=\toro^d$ on $M$ by translations
on the second factor. Our approach requires the existence of lifts 
\begin{equation*}
  S_i= 
  \partial_{\varphi_i} + \sum_{I=1}^d b_{iI}(\varphi)\, \partial_{\vartheta_I}
  \,,\qquad i=1,\ldots,k \,,
\end{equation*}
such that 
\begin{equation}\label{E:Lift0}
   [S_i, S_j]=0, \qquad [S_i, X] =0, \qquad \forall i,j=1,\dots , k.
\end{equation}
Note that it is not restrictive to assume that the $b_{iI}$ have zero
averages. Conditions \for{E:Lift0} are equivalent to
\begin{eqnarray}
\label{E:Lift1}
 &\frac{\partial b_{jI}}{\partial \varphi_i}- 
   \frac{\partial b_{iI}}{\partial \varphi_j} =0 
   \qquad \forall \, i,\,  j,\, I\, ,\\
 \label{E:Lift2}
 & \frac{\partial \xi_{I}}{\partial \varphi_j}-
    \sum_{i=1}^k \omega_i\frac{\partial b_{jI}}{\partial \varphi_i}=0
    \qquad  \forall\, j, \, I.
\end{eqnarray}

In this case system \for{triv} is
$\dot\varphi=\o $, $\dot\vartheta = \xi(\varphi)$
and, the group being abelian, the reducibility equation
\for{eq:reducibility} is
\begin{equation}\label{E:Lift3}
  \xi(\varphi) - \frac{\partial a}{\partial \varphi}(\varphi) \o = \const
\end{equation}
and implies
$$
  \frac{\partial \xi_I}{\partial \varphi_i}( \varphi) 
  - \sum_{j=1}^k 
  \frac{\partial^2 a_I}{\partial \varphi_i\, \partial \varphi_j}
  \o_j = 0 \qquad\forall \,i,I  \,.
$$
This last equation coincides with \for{E:Lift2} if $b_{iI}=
\frac{\partial a_I}{\partial \varphi_i}$ and this set of $b_{iI}$
satisfies the closure conditions \for{E:Lift1}. 

Conversely, assume that our conditions are satisfied. Since the
$b_{iI}$ have zero average, equation  \for{E:Lift1} implies the existence of
functions $a_{I}$ such that $b_{iI}=\frac{\partial a_I}{\partial
\varphi_i}$. Using the connectedness of the torus,
integrating equation \for{E:Lift2} gives the reducibility equation
\for{E:Lift3}.

\vskip2mm
{\it Remark: } {Proving} the existence of solutions $b_{iI}$ of
equations \for{E:Lift1}, \for{E:Lift2} is a different matter, that relies
on the convergence of series that involve small denominators and
depends on arithmetic properties of $\o_1,\ldots,\o_k$ and on
properties of the
Fourier expansions of the $\xi_I$. This observation is in \cite{Zung},
who derives equation \for{E:Lift2} in the simplest case $k=1$, $d=2$
as a condition for the integrability of $X$.

\section{Proof of the theorems}
\label{s:proof}

\subsection{Phases of commuting lifts of reduced periodic vector fields }
\label{s:proof0}
First we recall the definition of phase for the periodic  case,
from \cite{field2,hermans,CDS}. We say that a $G$-invariant
vector field $S\in\cX(\cF)$ has {\it periodic reduced flow} if the
reduced vector field $\base{S}:=\pi_*(S)$ has periodic flow with
positive smooth period function $\base{p}:\pi(\cF)\to\reali_+$; we
call $p:=\base p\circ \pi$ the {\it lifted period} of $S$. 
The action of $G$ on $\cF$ will be denoted by a dot. 

If $S\in\cX(\cF)$ has periodic reduced flow with lifted period $p$, then
for any $m\in \cF$, $\P^{S}_{p(m)}(m)$ belongs to the $G$-orbit of
$m$ and, given that the action is free, 
$$
  \P^{S}_{p(m)}(m)=\gamma(m).m 
$$
for a unique element $\gamma(m)\in G$. This defines a map
$\gamma:\cF\to G$ that we call {\it phase} of~$S$ (monodromy and
shift are also used). It is known that this map is smooth, has the
$G$-equivariance property
\begin{equation*}
   \gamma(g. m) = g\gamma(m) g^{-1} \qquad \mathrm{for\ all} \ g\in G
\end{equation*}
and is constant along the flow of $S$:
$\gamma(\Phi_t^{S}(m))=\gamma(m)$ for all $t\in\reali$ and $m\in \cF$,
see e.g.~\cite{fasso-giacobbe}.

We now consider phases of commuting vector fields:

\begin{lemma}\label{l:comm} Let $S$ and $S'$ be two $G$-invariant
commuting vector fields
on $\cF$ with periodic reduced flows. Then:

i.  For each $m\in \cF$, their phases $\gamma(m),\gamma'(m)$ commute.

ii. The phase of $S$ is constant along the flow of $S'$.

\end{lemma}

\begin{proof}
i. If $m\in \cF$, then
$$
  \Phi^{S'}_{p'(m)}( \Phi^{S}_{p(m)}(m)) = 
  \Phi^{S'}_{p'(m) }( \gamma (m).m) = 
  \gamma(m). \Phi^{S'}_{p'(m)}(m) = 
  \gamma (m)\gamma'(m).m 
$$
where the second equality follows from the $G$-invariance of $S'$ and
the fact that
the lifted period $p'$ of $S'$ is constant along $G$-orbits.
Since the flows of $S$ and $S'$ commute we conclude that
$\gamma(m)\gamma'(m).m = \gamma'(m)\gamma(m).m$; by the freeness of
the action, $\gamma'(m) \gamma(m) = \gamma(m)\gamma'(m)$.

ii. Let $m'= \P^{S'}_t(m)$ for some $t\in\reali$. 
Since the projected vector fields $\pi_*S$ and $\pi_*S'$ commute,
the lifted period of
$S$ is constant along the flow of $S'$. Thus $p(m')=p(m)$ and
$\gamma(m').m' 
= \P^{S}_{p(m')}(m')
= \P^{S'}_t( \P^{S}_{p(m)}(m)) 
= \P^{S'}_t( \gamma(m).m) 
= \gamma(m).\P^{S'}_t( m) 
= \gamma(m).m'$, so that $\gamma(m')=\gamma(m)$.
\end{proof}

\subsection{Proof of Theorem 1}
\label{s:proof1}
As noted in Section
\ref{section2}, see equation (\ref{ultimafase}), under the hypotheses
of Theorem~1 there exist $G$-invariant lifts
$S_1,\ldots,S_k\in\cX(\cF)$ of
$\partial_{\p_1},\ldots,\partial_{\p_k}$ which pairwise commute
and are such that
$$
  X = \o_1 S_1+\ldots+ \o_kS_k \,,
$$
and hence commute with $X$ as well. Thus
$$
  \P^X_t \ug \P^{S_1}_{t\o_1}\circ \cdots \circ\P^{S_k}_{t\o_k} 
  \qquad \forall t\in\reali \,.
$$
We will write $\{S\}$ for the collection $\{S_1,\ldots,S_k\}$ of
the lifts and
$$
  \P^{\{S\}}_x:=
  \P^{S_1}_{x_1}\circ\cdots\circ\P^{S_k}_{x_k} \qquad
  \forall \ x=(x_1,\ldots,x_k) \in\reali^k \,.
$$
Thus $\P^X_t=\P^{\{S\}}_{\o t}$, with $\o=(\o_1,\ldots,\o_k)$.

Fix a point $\m\in\cF$. The vector fields $S_1,\ldots,S_k$ are
$G$-invariant and have reduced periodic flows with unit period. Let
$\bar\gamma_1,\ldots,\bar\gamma_k$ be their phases at the point $\m$,
$\bar\gamma_i=\gamma_i(\m)$ if $\gamma_i$ is the phase of $S_i$. Due
to the commutativity of the phases, see Lemma 1, the set
$$
 \Gamma_2 := \{\bar\gamma_1^{z_1}\cdots\bar\gamma_k^{z_k} \,:\;
  (z_1,\ldots,z_k)\in\interi^k \}
$$
is an abelian subgroup of $G$. Let $\Tdue$ be a torus in $G$
that contains the closure of $\Gamma_2$. (Recall that a
torus in a Lie group $G$ is any connected compact abelian Lie
subgroup of $G$; all tori of $G$ have dimension not greater than the
rank of $G$, which is in fact the dimension of the `maximal tori' 
\cite{2003.Broker.tom Dieck, field3}). 
Let $\tdue\subseteq\g$ be the Lie algebra of $\Tdue$. Then, there exist
vectors $\eta_1,\ldots,\eta_k\in\tdue$ such that
$\exp(\eta_i)=\bar\gamma_i$ for $i=1,\ldots,k$. For economy of exposition,
we will call these vectors `logarithms' of the phases at $\m$;
they are not unique (being defined up to the addition of elements of
$\tdue$ that exponentiate to the group identity),
so we make a choice of them. We will write
$\eta=(\eta_1,\ldots,\eta_k)$ and, if $x\in\reali^k$, $x\star \eta$
for $x_1\eta_1+\ldots +x_k\eta_k$.

\begin{lemma}\label{l:j_m}
The map\footnote{
To keep the notation simple, we ordinarily avoid distinguishing
between a point $\a\in\reali^k$ and the equivalence class
$\a(\mathrm{mod1})\in\toro^k$. For instance, here, the right hand
side of the definition of the map $j$ is a function on
$\reali^k\times G$; the correctness of this
definition is checked in the proof.} 
\begin{equation}\label{j}
  j: \toro^k\per G \to \cF \,,\qquad
  j(\alpha,g) 
  = g\,\exp(-\alpha\star \eta) . \Phi^{\{S\}}_\alpha(\m) 
\end{equation}
is a diffeomorphism.
\end{lemma}

\begin{proof} 
The map $j$ is well defined because $j(\alpha+z,g)=j(\alpha,g)$ for
all $z\in \mathbb{Z}^k$. In fact
$$
 j(\alpha+z,g)=g\exp(-\a\star\eta)\exp(-z\star\eta)
.\P^{\{S\}}_\a(\P^{\{S\}}_z(\m)) = j(\alpha,g)
$$
given that, by the $G$-invariance of the lifts,
$$
 \P^{\{S\}}_\a(\P^{\{S\}}_z(\m)) = 
 \P^{\{S\}}_\a(\bar\gamma_1^{z_1}\ldots\bar\gamma_k^{z_k} .\m) = 
 \exp(z\star\eta) . \P^{\{S\}}_\a(\m) \,.
$$
To prove that $j$ is a diffeomorphism it suffices to show that
it is a local diffeomorphism at each point, and that it is bijective.

To prove that $j$ is a local diffeomorphism, consider a vector $(a,
\xi)\in \reali^k\times \g$. Let $v\in T_{(\a,g)}(\toro^k\times G)$ be
the vector which is tangent at $t=0$ to the curve $t\mapsto
(\alpha+ta,\exp(t\xi)g)$. Since the image via $j$ of this curve is
$$
  t\mapsto \exp(t\xi) g \exp(-(\alpha+at)\star\eta) .
  \P^{\{S\}}_{\alpha + t a}(\m) \,,
$$
by differentiation one finds
$$
   T_{(\a,g)}j\cdot v=X_\xi(j(\a,g)) - \sum_i a_iX_{\eta_i}(j(\a,g))  + 
   \sum_ia_iS_i(j(\a,g))
$$
where $X_\xi,X_{\eta_j}$ are the infinitesimal generators of the
group action associated to the elements $\xi,\eta_i\in\g$. Since the
$S_i$'s are not tangent to the orbits of $G$, the vanishing of
$T_{(\a,g)}j\cdot v$
requires $a=0$ and hence the vanishing of $X_\xi$; thus $\xi=0$.

To prove injectivity assume that $j(\alpha,g)=j(\beta,h)$. Then
\begin{equation}\label{inj}
  \P^{\{S\}}_{\a-\beta}(\m) 
  = 
  \exp(\a\star\eta)g^{-1}h\exp(-\beta\star\eta).\m
\end{equation}
and $\P^{\{S\}}_{\a-\beta}(\m)$ must be in the $G$-orbit of $\m$.
This happens if and only if $\alpha-\beta\in\interi^k$. Hence 
$
 \P^{\{S\}}_{\a-\beta}(\m) 
 = 
 \bar\gamma_1^{\a_1-\beta_1} \ldots \bar\gamma_k^{\a_k-\beta_k}.\m
 = 
 \exp\big((\a-\beta)\star \eta\big).\m
 = 
 \exp(\a\star \eta) \exp(-\beta\star \eta).\m
$
and equality \for{inj} reduces to 
$
 \exp(\a\star \eta).( \exp(-\beta\star \eta).\m)
 = 
 \exp(\a\star\eta)g^{-1}h.(\exp(-\beta\star\eta).\m
$, 
which gives $g^{-1}h=e_G$, the identity of the group.
Thus $\a\equiv\beta$ ($\mod 1$) and $g=h$.

Finally, for any $m\in \cF$ there exist $\beta\in \toro^k$ such that
$\Phi_{\beta}^{\{\hat S\}}(\pi(m))=\pi(\m)$. Hence
$\Phi_{\beta}^{\{\hat S\}}(\pi(m))=h.\m$ for some $h\in G$
and $m = j (\beta, h\exp(\beta\star \eta))$. Thus, $j$ is surjective.
\end{proof}

\vskip4mm
Note now that $\m=j(0,e_G)$, where $e_G$ is the identity of $G$, 
and
\begin{equation}\label{flusso-su-j}
  \P^X_t(j(\a,g)) \ug j(\a+t \o, g\exp(t\o\star \eta)) 
  \qquad \forall t,\, \a,\, g \,.
\end{equation}
The one-parameter subgroup 
$$
  \Gamma_1:=\{\exp(t\o\star\eta):\, t\in\reali\}
$$
of $G$ is an abelian and connected subgroup of $T_2$.
Hence, its closure is a torus
$\Tuno\subseteq \Tdue$ of some dimension ${d_1}\le{d_2}$.

The set
$$
  \cP_{\m}:=
  j(\toro^k\times \Tuno )
$$
is diffeomorphic to $\toro^k\per \Tuno$ and is invariant under the flow
of $X$: from \for{flusso-su-j}, if $(\beta,h)\in\toro^k\times \Tuno$ then
$
 \P^X_t(j(\beta,h)) = 
 j(\beta+\o t,h\exp(t\o \star \eta)) \in \cP_\m
$
because both $h$ and $\exp(t\o \star \eta)$ belong to the
subgroup $\Tuno$. If $m\in \cF$ then $m=j(\beta,g)$ for unique
$\beta\in\interi^k$, $g\in G$ and we define
$$
  \cP_m := g.\cP_\m \,.
$$

\begin{lemma}\label{l:P_m}
\ 
\bList
\item[i.] For any $m\in\cF$, $\cP_{m}$ is diffeomorphic to
$\toro^k\per \Tuno$ and is $X$-invariant.
\item[ii.] For any $m,m'\in \cF$, the sets $\cP_{m}$ and $\cP_{m'}$
are either equal or disjoint.
\item[iii.] The sets $\cP_{m}$, $m\in\cF$, are the fibers of a
fibration of $\cF$ with base diffeomorphic to~$G/\Tuno$. 
\eList
\end{lemma}

\newcommand{\CartanAlg}{\mathfrak t}

\begin{proof} 
i. Since $g.j(\beta,h)= gh\exp(-\beta\star\eta).\P^{\{S\}}_\beta(\m)=
j(\beta,gh)$,
$$
  \cP_{g.\m} \ug j(\toro^k\per g\Tuno) \,. 
$$
The $X$-invariance of $g.\cP_\m$ follows from the $X$-invariance of
$\cP_\m$ and from the $G$-invariance of $X$: if $(\beta,h)\in
\toro^k\per \Tuno$ then
$\P^X_t(g.j(\beta,h))=g.\P^X_t(j(\beta,h))\in g.\cP_\m$.

ii. Since $\cP_{g.\m} = j(\toro^k\per g\Tuno)$ this follows from the
facts that $j$ is a diffeomorphism of $\toro^k\per G$ onto $\cF$
and that lateral classes in a group either coincide or are disjoint.

iii. This follows from the fact that
the sets $\toro^k\per g\Tuno$ are the fibers of a fibration of
$\toro^k\per G$ with base $G/\Tuno$.
\end{proof}

We now prove that the restriction of the flow of $X$ to each set
$\cP_m$, $m\in\cF$, is conjugate to one and the same linear flow on
$\toro^{k+{d_1}}$, where ${d_1}$ is the dimension of~$\Tuno$. Choose
an integral basis $\xi=\{\xi_1,\ldots,\xi_{d_1}\}$ of the Lie
algebra $\CartanAlg_{1}\subseteq\g$ of $\Tuno$; `integral basis'
means that it generates the lattice of elements that exponentiate to
the unity: for $\zeta\in\CartanAlg_{1}$, $\exp\zeta=e_G$ if and only
if $\zeta=\sum_i z_i\xi_i$ with all $z_i\in\interi$. Then,
$\o\star\eta=\sum_{i=1}^{d_1}\nu_i\xi_i$ or
\begin{equation}\label{omega*eta=nu*xi}
  \o\star\eta \ug \nu\star\xi
\end{equation}
for some $\nu=(\nu_1,\ldots,\nu_{d_1})$. Moreover, for any $g\in G$, the map
$$
  i_g \;:=\; \toro^k \times \toro^{d_1} \;\to\; \cP_{g.\m} \,,
  \qquad i_g(\a,\beta)= j(\a,g\exp(\beta\star \xi)) 
$$
is a diffeomorphism.

\begin{lemma} \label{l:dynamics} For any $g\in G$, $i_g$ conjugates
the linear flow
$$
 (t,(\a,\beta)) \mapsto (\a+t\o, \beta + t \nu ) \qquad
 (\mod 1) 
$$
on $\toro^k\per\toro^{d_1}$ to the restriction of the flow of $X$ to
$\cP_{g.\m}$.
\end{lemma}

\begin{proof}
$
 \P^X_t(i_g(\a,\beta))
 =
 \P^X_t(j(\a,g\exp(\beta\star\xi))
 =
 j(\a+t\o , g\exp(\beta\star\xi+t\o \star\eta))
$. 
This equals $i_g(\a+t\o,\beta+t\nu )$ given that $\o\star\eta =
\nu\star\xi$. 
\end{proof}

This completes the proof of Theorem \ref{th1}, with $d={d_1}$ and $T=T_1$.

\acapo
{\it Remarks: } 1. The external frequencies $\nu$
are related to the internal ones by a linear relation, see
\for{omega*eta=nu*xi}. This might be
seen as related to the linearity of the reconstruction
equation.

2. The choice of the lifts of the generators, of their phases
and obviously of the point $\m$ are not unique. The entire
construction depends on them. Only the invariant tori of $\cF$ of
{\it minimal} dimension, being uniquely defined by the dynamics as
closure of trajectories, are independent of these choices.

3. What we need in the proof of Theorem 1 are the logarithms
$\eta_i$ of the phases $\bar\gamma_i$ and the fact that they
commute. This implies that they belong to some abelian Lie subalgebra
of $\g$, and the torus
$\Tdue$ is used only to define such a subalgebra. The dimension of
$\Tdue$ is immaterial to our construction, and even if it might seem
natural to choose it as small as possible, there is no need to
do that within the proof of Theorem \ref{th1}. The reason is that
we might need to
change anyway this choice in the proof of Theorem \ref{th2}, to remove
resonances between the internal and the external frequencies. 

\subsection{Proof of Theorem 3}
\label{s:proof3}
{\it i. $\implies$ ii.} Under hypothesis i., there exists an
equivariant diffeomorphism $D_1$ between $\cF$ and $\toro^k\times G$.
(`Equivariant' has here the same meaning as in Section \ref{s:3.1}).
By the reducibility of the reconstruction equation there is a
diffeomorphism  $D_2:\cF\to\toro^k\times G$ that conjugates the
vector field $X$ to the system \for{triv2} on $\toro^k\times G$,
namely
\begin{equation}
\label{triv2bis}
   \dot \varphi = \omega \,,\qquad \dot h = h \rho
\end{equation}
with a constant $\rho\in\mathfrak g$. The diffeomorphism $D_2$ is the
composition of $D_1$ and of the diffeomorphism $(\varphi,g)\mapsto
(\varphi,ga(\varphi)^{-1})$ as in \for{diffeo}, and is thus equivariant. System
\for{triv2bis} is invariant under the action $\L$ of $\toro^k\times G$
on itself by left translations,
$$
  \L_{(\a,g)}(\beta,h)=(\a+\beta,gh) \,.
$$
Therefore, $X$ is invariant under the action $\tilde\L$ of
$\toro^k\times G$ on $\cF$ given by
$$
   \tilde\L_{(\a,g)} := D_2^{-1}\circ \L_{(\a,g)}\circ\ D_2
$$
and this action is free because $\L$ is free. By the equivariance of
$D_2$, the restriction
of the action $\tilde\L$ to the subgroup $\{0\}\times G$ coincides
with the action $\Psi$ of $G$ on $\cF$. Hence,
the restriction of $\tilde\L$ to the subgroup
$\toro^k\times\{e_G\}$ gives an
action $\chi$ of $\toro^k$ on $\cF$ that has the properties of condition ii.
of Theorem 3.

{\it ii. $\implies$ iii.}    Let $Y_1,...,Y_k\in\cX(\cF)$ be a set of
infinitesimal generators of the action $\chi$ of $\toro^k$ on $\cF$.
Since $\chi$ commutes with $\Psi$, the vector fields $Y_i$ are
$\pi$-related to vector fields $\widehat Y_i$ in $\widehat{\cF}$
(with, of course, $\pi:\cF\to\cF/G$ the canonical projection). The
$\chi$-invariance of $X$ and the density of its trajectories on $\cF/G$
implies that the vector fields $\widehat Y_i$ have the form
$\sum_{j=1}^kr_{ij}\partial_{\varphi_j}$ with constant $r_{ij}$.
The freeness hypothesis implies that the vector fields $\widehat Y_i$ are
independent. Hence, there exist constant coefficients $s_{ij}$ such
that, for every $i$, $\partial_{\varphi_i} = \sum_{j=1}^ks_{ij}
\widehat Y_j$. The $k-1$ vector fields $S_i = \sum_{j=1}^k
s_{ij} Y_j$, $i=1,...,k-1$, have the properties of condition
iii. 

{\it iii. $\implies$ i.}  Under hypothesis iii., the construction
done in the proof of Theorem 1 is valid. In particular, by Lemma
\ref{l:j_m}, $\cF$ is diffeomorphic  the map
$j:\toro^k\times G \to \cF$ as in \for{j} is an equivariant
diffeomorphism, and hence $\cF$ is a trivial $G$-bundle. Formula
\for{flusso-su-j} shows that $j$ conjugates $X$ to the constant
system
$$
   \dot \varphi = \o \,, \qquad
   \dot g = g\, \o\star\eta
$$
on $\toro^k\times G$. Hence, the reconstruction equation is reducible
(with $\rho=\o\star\eta$).

\acapo
{\it Remark: } In the hypotheses of Theorem 1,  the additional action
$\chi$ of $\toro^k$ is constructed as follows. Consider the free
action $J$ of $\toro^k\times G$ on $\cF$ given by
$$
  J_{(\a,g)} := j\circ \L_{(\a,g)}\circ\ j^{-1}
$$
From
$$
  J_{(\a,g)} (j(\beta,h)) = j(\a+\beta,gh) 
$$
and \for{flusso-su-j} it follows that
$\P^X_t(J_{(\a,g)}(j(\beta,h))) = j(\a+\beta+\o t, g h
\exp(t \o \star \eta)) = J_{(\a,g)}(\P^X_t(j(\beta,h))) $, which
shows that $X$ is $J$-invariant. Moreover, $J|_{\{0\}\times G}=\Psi$
and $\chi:=J|_{\toro^k\times \{e_G\}}$ is an action of $\toro^k$ on
$\cF$ with all the properties of condition ii.

\subsection{Proof of Theorem 2}
\label{s:proof2}

In the proof of Theorem 2 we use the entire construction done
in the proof of Theorem~1. If the frequency vector
$(\o,\nu)$ constructed in that proof
happens to be nonresonant then Theorem \ref{th2} is
valid with ${d_0}={d_1}$, $T=T_1$ and
$r=1$.
It remains to be considered the case in which the vector $(\o,\nu)$
is resonant.

A resonance of $(\o,\nu)\in\reali^{k+{d_1}}$ is a nonzero integer vector
$(p,q)\in\interi^{k+{d_1}}$ such that \hbox{$p\cdot\o+q\cdot\nu=0$}.
Resonances of $(\o,\nu)$ form a lattice $\Lambda$ of
$\interi^{k+{d_1}}$ of rank $l\ge1$. We recall that a lattice of
$\interi^m$ of rank $l$ is the set of all linear combinations with
integer coefficients of $l$ linearly independent vectors of
$\interi^m$, called a basis of the lattice. In our case, let
$$
  (\tilde p^1,\tilde q^1),\ldots,(\tilde p^l,\tilde q^l) 
  \in\interi^{k+{d_1}} 
$$
be a basis of the lattice $\Lambda$.

Remember that the vector $\omega\in\reali^k$ is nonresonant by
assumption. The density of the set 
$\{\exp({t\omega \star \eta}) = \exp(t\nu \star \xi):\, t\in\reali\}$ 
in the ${d_1}$-dimensional torus $\Tuno$ implies that the vector
$\nu\in\reali^{d_1}$ is nonresonant as well. Thus, each
of the two groups of vectors $\tilde p^1,\ldots,\tilde p^l\in\interi^k$ and 
$\tilde q^1,\ldots,\tilde q^l\in\interi^{d_1}$ forming the basis of
$\Lambda$ is
independent over $\interi$. (For, if $\sum_i c_i\tilde p^i=0$ with
$c_1,\ldots,c_l\in\interi$, then $0=\sum_i c_i(\tilde p^i\cdot \o +
\tilde q^i\cdot \nu) = \sum_i c_i\tilde q^i\cdot \nu$ and hence
$\sum_i c_i\tilde q^i=0$; but then
$\sum_i c_i(\tilde q^i,\tilde p^i)=0$ and all $c_i=0$).
From this it follows that $l\le\min(k,{d_1})$. Moreover, it is
possible to put the resonances in a simpler form:

\begin{lemma}\label{l:SNF} There is an integral basis\footnote{The
notion of integral basis is defined just after Lemma 3.}
$\xi'=\{\xi_1',\ldots,\xi'_{d_1}\}$ of $\tuno$ such that, if $\nu'\in
\reali^{{d_1}}$ denotes the components of $\nu\star\xi\in\tuno$
in this basis (that is, $\nu\star\xi=\nu'\star\xi'$), then the
resonant lattice $\L'\subset\interi^{k+{d_1}}$
of $(\o,\nu')\in\reali^{k+{d_1}}$ has a basis formed by $l$ vectors
\begin{equation}\label{NuovaBase}
 (p^1,r_1e^1)\,,\, \ldots \,,\,  (p^l,r_le^l) \in \interi^{k+{d_1}}
\end{equation}
where $p^1,\ldots,p^l \in \interi^k$, 
$e^i$ denotes the $i$-th vector of the standard basis of
$\interi^{d_1}$, and the $r_i$'s are positive integers with the
property that if $r_i<r_j$ then $r_i$ divides $r_j$.
\end{lemma}

\begin{proof}
Let $Q$ be the ${d_1}\times l$ integer matrix with columns $\tilde
q^1,\ldots,\tilde q^l$. By the Smith Normal Form Theorem (see e.g.
\cite{cohn}), there exist a $d_1\times d_1$ integer matrix $Z$ and
an $l\times l$ integer matrix $C$, both invertible over $\interi$
(that is, unimodular, or having determinant $\pm1$), such that the
matrix  $Z Q C^T$ has the block structure 
\begin{equation}\label{SmithNF}
  Z Q C^T 
  = 
 \begin{pmatrix} \mathrm{diag}(r_1,...,r_l)
                 \\
                 O_{{d_1}-l,l}
 \end{pmatrix} 
\end{equation}
where $O_{{d_1}-l,l}$ is the $({d_1}-l)\times l$ null block and
$r_1,\ldots,r_l$ are nonnegative
integers such that if $0\not=r_i<r_j$ then $r_i$
divides $r_j$. Since $Q$ has rank $l$, all $r_j\not=0$. 
Let $Z_{ij}\in\interi$ denote the entries of $Z$. Since $Z$ is
unimodular, the vectors
$$
  \xi'_i = \sum_{j=1}^{d_1} Z_{ij}\xi_j \,,\qquad i=1,\ldots,{{d_1}} \,,
$$
form a new integral basis $\xi'$ of $\tuno$ and
\begin{equation}\label{nuxi}
  \nu\star\xi = \nu'\star\xi' \qquad\mathrm{with} \ \nu'=Z^{-T}\nu
  \,.
\end{equation}
Since $(\o,\nu')\cdot(p,q) = \o\cdot p+\nu\cdot Z^{-1}q$, the resonant
lattice $\L'$ of $(\o,\nu')\in\reali^{k+{d_1}}$ has a basis formed by
the $l$ vectors $(\tilde p^i,Z\tilde q^i)$, $i=1,\ldots,l$. 
Let now $C_{ij}\in\interi$ be the entries of $C$. Since $C$ is unimodular,
another basis of $\L'$ is formed by the $l$ vectors 
$$
  \sum_{j=1}^lC_{ij}(\tilde p^j,Z\tilde q^j)
  = 
  \Big(\sum_{j=1}^lC_{ij}\tilde p^j, \sum_{j=1}^l C_{ij}Z\tilde q^j\Big)
  \,,\qquad i=1,\ldots, l  \,.
$$ 
This is the basis \for{NuovaBase}, with 
$p^i=\sum_{j=1}^lC_{ij}\tilde p^j$ for $i=1,\ldots,l$, because
the first $l$ rows of (\ref{SmithNF}) read
$\sum_{j=1}^l C_{ij} Z\tilde q^j = r_ie^i$, $i=1,\ldots, l$. 
\end{proof}

Note that the resonance relations satisfied by the new frequency vector
$(\o,\nu')\in\reali^{k+{d_1}}$ that correspond to the vectors of the
basis (\ref{NuovaBase}) are
\begin{equation}\label{resonances}
  p^i\cdot \o + r_i\nu'_i =0 \,,\qquad i=1,\ldots,l \,.
\end{equation}
Moreover, Lemma \ref{l:SNF} gives a decomposition of the Lie algebra
$\tuno$ as a sum of two subalgebras, $\tuno=\tres\oplus\tzero$, with 
$\tres=\mathrm{Span}(\xi'_{1},\ldots,\xi'_l)$ of dimension $l$ and
$\tzero=\mathrm{Span}(\xi'_{l+1},\ldots,\xi'_{d_1})$ of dimension
${d_0}={d_1}-l$. This decomposition is such that if
$\nu''\in\reali^{d_0}$
denotes the vector of the
components of the
frequency vector $\nu'\star\xi'$ in $\tzero$, that is
$\nu''=(\nu'_{l+1},\ldots,\nu'_{{d_1}})$, then the
vector $(\o,\nu'')\in\reali^{k+{d_0}}$ is nonresonant.

We now proceed as in the proof of Theorem \ref{th1}, but instead of the 
lifts $S_1,\ldots,S_k$ considered there we consider the lifts
$$
  S'_i = r S_i \,,\qquad i=1,\ldots,k \,,
$$
where $r=\max(r_1,\ldots,r_l)$ (see Lemma \ref{l:SNF}). Thus
$\P^{S'_i}_1(\m) = \P^{S_i}_{r}(\m) = \bar\gamma_i^r.\m$. 
Correspondingly, we define
$$
  \o'=\frac\o r 
$$
so that $X = \sum_{j=1}^k \o_j'S'_j$ and $\P^X_t =
\P^{\{S'\}}_{t\o'}$.
A possible choice of logarithms of the powers $\bar{\gamma}_i^r$ of 
the phases would be $r\eta_i$, but we choose instead
\begin{equation}\label{eta'}
  \eta'_i = r \eta_i + \d_i
  \quad \mathrm{with} \quad  
  \d_i = \sum_{j=1}^l \frac r{r_j} p_i^j \xi'_j \,, \qquad
  i=1,\ldots,k \,;
\end{equation}
we will write $\d=(\d_1,\ldots,\d_k)$. The need of correcting the
choice of logarithms with the elements $\delta_i$ will become clear
in the lemma immediately below: these elements of the Lie algebra are
precisely the corrections needed to make the map $j'$ a covering map.
Note that $\d_i\in\tres$ and $\exp\d_i=e_G$ (so that $\exp\eta'_i =
\bar\gamma_i^r$) because the $\d_i$ and $\eta_i$ commute and the
$p^j_i$ and $r/r_i$ are integers.

\begin{lemma}\label{l:j'_m} The map
\[
    j':\mathbb T^k \times G \to \cF \,,\qquad
    j'(\alpha,g) = 
    g \exp({-\alpha\star\eta'}). \Phi^{\{S'\}}_{\alpha}(\m) \,,
\]
is a smooth $r^k:1$ covering map.
\end{lemma}

\begin{proof} The same arguments used in the proof of Lemma \ref{l:j_m}
show that $j'$ is well defined, is surjective
and is a local diffeomorphism at each
point. Because of the latter property, if the cardinality of its
fibers is constant then it is a
covering map. Assume $j'(\alpha,g)=j'(\beta,h)$, or
\begin{equation}\label{inj'}
  \P^{\{S'\}}_{\beta-\a}(\m) 
  = 
  \exp(\beta\star\eta')h^{-1}g\exp(-\a\star\eta').\m \,.
\end{equation}
This implies $\P^{\{S'\}}_{\beta-\a}(\m)\in G.\m$. Since the
$S'_i$ have lifted period $1/r$, this happens
if and only if $\beta\equiv\a$ ($\mod\frac1{r}$) or
$$
  \beta \equiv \a + \frac1 r\, {u} \ (\mod 1)
$$
with some $u=(u_1,\ldots,u_k) \in \interi_{r}^k$, where
$\interi_r=\interi/r\interi \equiv \{0,\ldots,r-1\}$,
and in this case
$
 \P^{S'_i}_{\beta_i-\a_i}(\m) 
 = 
 \P^{S_i}_{u_i}(\m) 
 =
 \bar\gamma_i^{u_i}.\m \,.
$
So, equation \for{inj'} implies
$$
   \P^{\{S'\}}_{\beta-\a}(\m) = 
   \bar\gamma_1^{u_1} \cdots \bar\gamma_k^{u_k}.\m
   \qquad\mathrm{for\ some}\quad  u \in\interi_{r}^k \,.
$$
Thus, if equation (\ref{inj'}) is satisfied, there is
$u\in\interi_{r}^k$ 
such that $\bar\gamma_1^{u_1} \cdots \bar\gamma_k^{u_k}.\m
= \exp(\beta\star\eta')h^{-1}g\exp(-\a\star\eta').\m$, or
$$
 h^{-1}g 
  = \bar\gamma_1^{u_1} \cdots \bar\gamma_k^{u_k}
    \exp\Big(-\frac{u_1}{r}\eta'_1 - \ldots
             -\frac{u_k}{r}\eta'_k\Big) 
  = \exp\Big(- \frac1r u\star \d\Big) \,,
$$
with $u\star\d= \sum_{i=1}^ku_i\delta_i$.
In conclusion, $j'(\alpha,g)=j'(\beta,h)$ if and only if there exists
$u\in\interi_{r}^k$ such that
\begin{equation}\label{DeckTransformations}
  \beta = \a+\frac ur   \ (\mathrm{mod}1)
  \,,\qquad
  h = g 
  \exp\Big(\frac1r u\star \d \Big) \,.
\end{equation}
Thus, the cardinality of the fibers of $j'$ is $r^k$.
\end{proof}

By (\ref{DeckTransformations}), the deck transformations of the
covering $j'$ of $\cF$ give an action
$\Psi$ of $\interi_{r}^k$ on $\toro^k\times G$ defined by
$$
  \Psi_u(\a,g)
  =
  \Big( \a+\frac ur \,,\;
  g\exp\Big(\frac1r u\star\d \Big)
  \Big)
$$
that is free and satisfies $j'\circ\Psi_u=j'$ for all $u \in \mathbb
Z^k_r$ (we recall that the deck transformations are the maps of a
covering onto itself that preserve the fibers \cite{1991.Massey}).
Therefore:

\begin{lemma}\label{l:diffeo} $\cF$ is diffeomorphic to 
$(\toro^k\times G)/\interi_{r}^k$ (the quotient being relative to the
action $\Psi$).
\end{lemma}

Now
\begin{equation}\label{Flusso-e-j'}
  \Phi_t^X(j'(\alpha,g))
  \ug
  \Phi^{\{S'\}}_{t\o'} 
      \Big(g \exp({-\alpha\star\eta'}).\Phi^{\{S'\}}_{\a}(\m) \Big)
  \ug
  j'\big(\a+\o' t \,,\; g\exp(t \o' \star \eta')\big)
\end{equation}
and instead of the subgroup $\Gamma_1$ of the proof of Theorem
\ref{th1} we are lead to consider the subgroup
$$
   \Gamma_0:=\big\{\exp(t\o'\star\eta'):\, t\in\reali\big\} \,.
$$
By \for{resonances} and \for{eta'}, $\sum_{i=1}^k\o'_i\d_i =
\sum_{j=1}^l \frac{\o\cdot p^j}{r_j}\xi'_j = - \sum_{j=1}^l \nu'_j
\xi'_j$ and hence, using \for{omega*eta=nu*xi} and \for{nuxi},
$$
  \o'\star\eta' 
  \ug  \o\star \eta + \sum_{i=1}^k \o'_i\delta_i 
  \ug  \nu'\star \xi' - \sum_{j=1}^l \nu'_j \xi'_j
  \ug \sum_{j=l+1}^{d_1} \nu'_j\xi'_j  \,.
$$
Therefore, the closure of $\Gamma_0$ is the ${d_0}$-dimensional
torus $\Tzero=\exp(\tzero)$. For shortness, we write
$$
  \xi''_i=\xi'_{l+i} \,,\qquad 
  \nu''_i=\nu'_{l+i} \,,\qquad
  i=1,\ldots,{d_0} \,, 
$$
and $\xi''=(\xi''_1,\ldots,\xi''_{d_0})$, so that $\o'\star\eta' =
\nu''\star\xi''$.  Thus, \for{Flusso-e-j'} becomes
\begin{equation}\label{danotare}
  \Phi_t^X(j'(\alpha,g))  =
  j'\big(\a+\o't, g\exp(t \nu''\star\xi'')\big)  
\end{equation}
and the sets
$$
  \cP'_{g.\m} := j'(\toro^k\times g\Tzero)  \,,\qquad g\in G \,,
$$
are invariant under the flow of $X$ and foliate $\cF$. Consider the
subgroup $K$ of  $\interi_r^k$ defined as
\begin{equation}\label{K}
  K =\Big\{u\in\interi^k_r \,:\;
     \frac{u\cdot p^j}{r_j} \in\interi 
     \ \mathrm{for\ all\ } j=1,\ldots, l\Big\} \,.
\end{equation}

\begin{lemma} For each $g\in G$, $\cP'_{g.\m}$ is diffeomorphic to
$(\toro^k\times g\Tzero)/K$ (the quotient being relative to the
action $\Psi$).
\end{lemma}

\begin{proof} Clearly, it suffices to consider $g=e_G$. It follows
from Lemma \ref{l:diffeo} that $\cP'_{\m}$ is diffeomorphic to the
quotient of $\toro^k\times \Tzero$ by the subgroup of the deck 
transformations that map $\toro^k\times \Tzero$ to itself.
Thus, we only have to show that this
subgroup is $K$. Assume $(\a,h)\in \toro^k\times\Tzero$.
Since $\d_i\in\tres$ and $\Tzero=\exp(\tzero)$,
$\Psi_u(\a,h) = (\a+\frac ur ,h\exp\frac{u\star\d}r)$ is in 
$\toro^k\times\Tzero$ if and only if $\exp\frac{u\star\d}r=e_G$.
Since
$\frac{u\star\d}r =\sum_{i=1}^k \frac{u_i\d_i}r = 
\sum_{j=1}^l\frac{u\cdot p^j}{r_j}\xi'_j$
and the $\xi'_j$ are integral Lie algebra elements, this happens if
and only if  $\frac{u\cdot p^j}{r_j}\in\interi$ for all
$j$, or $u\in K$.
\end{proof}

Thus, the restriction of $j'$ to $\toro^k\times g\Tzero$ is a covering of
$\cP'_{g.\m}$ with fibers
$$
  \Psi_K(\a,h) = \{(\a+ r^{-1} u , h) \,:\; u\in K \} 
  \,,\qquad (\a,h)\in \toro^k\times g\Tzero \,.
$$
(Note however that the preimage under $j'$ of 
$\cP'_{g.\m}$ is $\toro^k\times
g\Tzero F_0$, where 
$
  F_0 = \{ \exp(r^{-1}u\star\d) \,:\; u\in\interi_r^k\} 
$).
We thus introduce coordinates on
these covering tori, with the map
$$
  i'_g : \toro^k\times\toro^{d_0}  \to \cP_{g.\m}' \,, 
  \qquad i'_g(\a,\beta) = j'\big(\a,g\exp (\beta\star\xi'')\big) 
$$
which clearly is still a covering map, with fibers diffeomorphic to
those of $j'$ (see below). From \for{danotare} it then follows that

\begin{lemma} For each $g\in G$, $i'_g: \toro^k \times \toro^{d_0}\to
\cP'_{g.\bar m}$ relates
the linear flow on $\mathbb T^k \times \mathbb T^{d_0}$  with
frequencies $(\omega',\nu'')$ to the flow of $X$ on $\cP'_{g.\m}$.
\end{lemma}

As already remarked, the frequency vector $(\o',\nu'')$ is
nonresonant.

\begin{lemma}
The sets $\cP'_{g.\m}$ are the fibers of a fibration of $\cF$ whose
base is diffeomorphic to $G/(\Tzero F_0)$, with $F_0 = \{
\exp(r^{-1}u\star\d) \,:\; u\in\interi_r^k\} \equiv \interi_r^k/K$.
\end{lemma}

\begin{proof}
The quotient of $\cF$ by the sets $\cP_{g.\m}'$ is diffeomorphic to
the quotient of $\toro^k\times G$ by the preimages under $j'$ of such
sets. The preimage of $\cP_{g.\m}'$ is $\toro^k\times gT_0F_0$. 
\end{proof}

This proves Theorem 2 with $T=T_0F_0$. Since $T_0$ and $F_0$ are
abelian subgroups contained in $T_1$, and their intersection is $e_G$,
$T_0F_0$ is abelian and diffeomorphic to $T_0\times F_0$.

\acapo
{\it Remarks: } 1. Our treatment of resonances differs from that in
\cite{CDS}, which is restricted to the case
$k=1$. Following that approach,
in the proof of Theorem \ref{th1} we would have
taken $T_2$ as the smallest compact subgroup of $G$ that contains
the closure of $\Gamma_2$ (see also Remark 3 at the end of Section
\ref{s:proof1}).
If $T_2$ is connected, and hence a torus, when $k=1$ there is nothing
else to do. If $T_2$ is not connected, then it is the product of a
torus $T$ and of a finite group $F$, say of order $r_1$. When $k=1$,
multiplying by $r_1$ the unique lift $S_1$ yields the power
$\bar\gamma_1^{r_1}$ of the unique phase
$\bar\gamma_1$ and automatically eliminates all
resonances between the internal frequency $\o_1$ and the
external frequencies $\nu$. When $k\ge2$, however, this procedure
does not eliminate the possibility of resonances between the
internal and the external frequencies.

2. The description of motions given in statement i. of Theorem 2 takes
place not in the torus $\cP'_m \subset \cF$, but in a
torus $\mathbb T^k \times g T_0$ that covers it. The frequencies
$(\omega',\nu'')$ are relative to the motion in such longer torus. In
order to describe the motion in the tori $\cP'_m$ one should
introduce angles on them. The relation between the chosen angles in
$\mathbb T^k \times g T_0$ and those in given $\cP'_m$ can be
obtained writing the group of deck transformations $K$ defined in
\for{K} in a normal form. This requires the computation of a unimodular,
integer, $k\times k$ matrix $M$ which makes $K = s_1 \mathbb Z_r
\times \cdots \times s_l \mathbb Z_r \times \mathbb Z_r \times \cdots
\times \mathbb Z_r$. It hence follows that the frequencies of the
motion in $\cP_m$ are $(\mu,\nu'')\in \mathbb R^{k+d_0}$ with
$\mu = D M \omega'$, where $D$ is the diagonal $k\times k$
matrix whose entries are $s_1,...,s_l,1,...,1$, $l = d-d_0$, and each
$s_i$ divides the integer $r_i$ introduced in the proof of Theorem 2
(Lemma \ref{l:SNF}).
\acapon

\acapon
{\footnotesize
{\bf Acknowledgements. } This work is part of the research project
CPDR129597
{\it Symmetries and integrability of nonholonomic mechanical systems}
of the University of Padova. The authors wish to thank in a special
way one of the anonymous referees, whose very competent, pertinent,
stimulating comments have significantly improved the article,
motivating, in particular, the formulation of Theorem 3.
The authors thank Nicola
Sansonetto, who contributed to a preliminary stage of this work,
\`Angel Jorba and James Montaldi for useful conversations
on these topics, and an anonymous member of the Editorial Board for
her/his preliminary comments on the article. 
LGN would like to thank the Mathematics Department of the University
of Padova for its hospitality in the occasion of several visits,
during which part of this work  was done and acknowledges the support received from the project PAPIIT IA103815. AG acknowledges the support of ``Progetto Giovani GNFM 2015".
}

\end{document}